\documentclass[11pt]{amsproc}

\usepackage{amsmath, amssymb, amsthm, latexsym,nicefrac, setspace}
\usepackage[pagebackref,colorlinks,linkcolor=red,citecolor=blue,urlcolor=blue,hypertexnames=true]{hyperref}

\usepackage{amsrefs}

\theoremstyle{plain}
\newtheorem{theorem}{Theorem}[section]

\newtheorem{corollary}[theorem]{Corollary}

\newtheorem{definition}[theorem]{Definition}
\newtheorem{conjecture}[theorem]{Conjecture}

\theoremstyle{remark}
\newtheorem{remark}[theorem]{Remark}

\input{xy}
\xyoption{all}

\def\im{{\rm im}}
\def\dim{{\rm dim}}
\def\Z{\mathbb Z}

\def\C{\mathbb{C}}

\def\N{\mathbb{N}}

\title{A note on normal generation and generation of groups}

\author{Andreas Thom}
\address{A.T., Mathematisches Institut, U Leipzig,
PF 100920, 04009 Leipzig, Germany}
\email{andreas.thom@math.uni-leipzig.de}

\begin{document}

\onehalfspace

\begin{abstract} In this note we study sets of normal generators of finitely presented residually $p$-finite groups. We show that if an infinite, finitely presented, residually $p$-finite group $G$ is normally generated by $g_1,\dots,g_k$ with order $n_1,\dots,n_k \in \{1,2,\dots \} \cup \{\infty \}$, then
$$\beta_1^{(2)}(G) \leq k-1-\sum_{i=1}^{k} \frac1{n_i},$$
where $\beta_1^{(2)}(G)$ denotes the first $\ell^2$-Betti number of $G$. We also show that any $k$-generated group with $\beta_1^{(2)}(G) \geq k-1-\varepsilon$ must have girth greater than or equal $1/\varepsilon$.
\end{abstract}

\maketitle

\section{Introduction}

In the first part of this note we want to prove estimates of the number of normal generators of a discrete group in terms of its first $\ell^2$-Betti number. It is well-known that if a non-trivial discrete group is generated by $k$ elements, then 
\begin{equation} \label{triv}
\beta_1^{(2)}(G) \leq k-1.
\end{equation} The proof of this statement is essentially trivial using the obvious Morse inequality. The following conjecture was first formulated in \cite{thomosin}.

\begin{conjecture} \label{conj1}
Let $G$ be a torsionfree discrete group. If $G$ is normally generated by elements $g_1,\dots,g_k$, then
$$\beta_1^{(2)}(G) \leq k-1.$$
\end{conjecture}

If $G$ is finitely presented, residually $p$-finite for some prime $p$, then Conjecture \ref{conj1}, i.e., the inequality $\beta_1^{(2)}(G) \leq k-1$ is known to be true, see Remark \ref{lack}.
In this note, we give a proof of a variation of this conjecture, which also applies to the non-torsionfree case. In Theorem \ref{main1} we show that if $G$ is an infinite, finitely presented, residually $p$-finite group for some prime $p$, and if $G$ is normally generated by elements $g_1,\dots,g_k \in G$ of order $n_1,\dots,n_k \in \{1,2,\dots \} \cup \{\infty\}$, 
then 
\begin{equation}
\beta_{1}^{(2)}(G)  \leq k - 1 - \sum_{i=1}^k \frac1{n_i}.
\end{equation}

The proof is based on some elementary calculations with cocycles on $G$ taking values in $\C[G/H]$, for $H \subset G$ a normal subgroup of finite index, and L\"uck's Approximation Theorem \cite{lueckl2}.

In Section \ref{uncertsec} we prove that if a $k$-generated group $G$ satisfies $\beta_1^{(2)}(G) \geq k-1-\varepsilon$, then the shortest relation in terms of the generators must have length at least $1/\varepsilon$. A theorem of Pichot \cite{pichot} already implied that the girth of the Cayley graph of $G$ with respect to the natural generating set becomes larger and larger if $\varepsilon$ is getting smaller and smaller. Our main result is a quantitative estimate that implies this qualitative result.

\section{Residually $p$-finite groups}

In this section we want to recall some basic results on the class of residually $p$-finite groups and show that various natural classes of groups are contained in this class of groups. Let us first recall some definitions.

\begin{definition} Let $p$ be a prime number.
A group $G$ is said to be residually $p$-finite, if for every non-trivial element $g \in G$, there exists a normal subgroup $H \subset G$ of p-power index such that $g \not \in H$. A group $G$ is called virtually residually $p$-finite if it admits a residually $p$-finite subgroup of finite index.
\end{definition}

The following result relates residually $p$-finiteness to residual nilpotence and gives a large class of examples of groups which are residually $p$-finite.

\begin{theorem}[Gruenberg]
Let $G$ be a finitely generated group. 
If $G$ is torsionfree and residually nilpotent, then it is residually $p$-finite for any prime $p$.
\end{theorem}

Another source of residually $p$-finite groups is a result by Platonov, see \cite{MR0231897}, which says that any finitely generated linear group is virtually residually $p$-finite for almost all primes $p$.
In \cite{MR3100378}, Aschenbrenner-Friedl showed that the same is true for fundamental groups of $3$-manifolds.
Gilbert Baumslag showed \cite{MR0212078} that any one-relator group where the relator is a $p$-power is residually $p$-finite.

We denote the group ring of $G$ with coefficients in a ring $R$ by $RG$. Its elements are formal finite linear combinations of the form $\sum_g a_g g$ with $a_g \in R$. The natural multiplication on $G$ extends to $RG$. The natural homomorphism $\varepsilon \colon RG \to R$ given by 
$$\varepsilon\left(\sum_{g \in G} a_g g\right) := \sum_{g \in G} a_g$$
is called augmentation. We denote by $\omega_R$ the kernel of $\varepsilon \colon RG \to R$; the so-called augmentation ideal.

In the proof of our main result, we will use the following characterization of finite $p$-groups that was obtained by Karl Gruenberg, see \cite{MR0159881} and also \cites{MR0068540, MR0272895}, will play an important role.

\begin{theorem}[Gruenberg] \label{gr1}
Let $G$ be a finite group and let $\omega_{\Z} \subset \Z G$ be the augmentation ideal. The group $G$ is of prime-power order if and only if 
$$\bigcap_{n=1}^{\infty} \omega^n = \{0\}.$$
\end{theorem}

%It is easy to see that this implies that the augmentation ideal of $\Z G$ for any residually $p$-finite group has the same property, i.e., $\cap_{n=1}^{\infty} \omega^n = \{0\}$.
%
%\begin{theorem} Let $G$ be a residually $p$-finite group, which is normally generated by some subgroup $\Lambda$. If $G$ has infinitely many ends, the $\Lambda$ also has infinitely many ends.
%\end{theorem}
%

\section{$\ell^2$-invariants of groups}

\subsection{Some definitions} $\ell^2$-invariants of fundamental groups of compact aspherical manifolds where introduced by Atiyah in \cite{atiyah}. A definition which works for all discrete groups was given by Cheeger-Gromov in \cite{cheegro}. Later, a more algebraic framework was presented by L\"uck in \cite{lueckl2}. We want to stick to this more algebraic approach. 

Let $G$ be a group and denote by $\C G$ the complex group ring. Note that the ring $\C G$ comes with a natural involution $f \mapsto f^*$ which is given by the formula$$(\sum_{g \in G}a_g g)^{*} = \sum_{g \in G} \bar a_g g^{-1}.$$
We denote by $\tau \colon \mathbb \C G \to \mathbb C$ the natural trace on $\mathbb C G$, given by the formula $$\tau\left(\sum_{g \in G} a_g g \right) = a_e.$$ It satisfies $\tau(f^*f) \geq 0$ for all $f \in \C G$ and the associated GNS-representation is just the Hilbert space $\ell^2 G$ with orthonormal basis $\{\delta_g \mid g \in G\}$ on which $G$ (and hence $\C G$) acts via the left-regular representation. More explicitly, there exists a unitary representation $\lambda \colon G \to U(\ell^2 G)$ and $\lambda(g) \delta_h = \delta_{gh}$.
Similar to the left-regular representation, there is a right-regular representation $\rho \colon G \to U(\ell^2 G$), given by the formula $\rho(g) \delta_h = \delta_{hg^{-1}}$.

The group von Neumann algebra of a group is defined as
$$L G := B(\ell^2 G)^{\rho(G)} = \{T \in B(\ell^2 G) \mid \rho(g)T=T\rho(g), \forall g \in G\}.$$
It is obvious that $\lambda(\C G) \subset LG$, in fact it is dense in the topology of pointwise convergence on $\ell^2G$.
Recall that the trace $\tau$ extends to a positive and faithful trace on $LG$ via the formula
$$\tau(a) = \langle a \delta_e, \delta_e \rangle.$$
For each $\rho(G)$-invariant closed subspace $K \subset \ell^2 G$, we denote by $p_K$ the orthogonal projection onto $K$. It is easily seen that $p_K \in LG$. We set $\dim_G K := \tau(p_K) \in [0,1]$. The quantity $\dim_GK$ is called Murray-von Neumann dimension of $K$. L\"uck proved that there is a natural dimension function
$${\rm \dim_{LG}}\colon LG \mbox{-modules } \to [0,\infty]$$
satisfying various natural properties, see \cite{lueckl2}, such that $\dim_{LG} K = \dim_G K$ for every $\rho(G)$-invariant subspace of $\ell^2G$.
 
\vspace{0.1cm}

We can now set
$$\beta_1^{(2)}(\Gamma):= \dim_{L\Gamma} H^1(\Gamma,L\Gamma),$$
where the group on the right side is the algebraic group homology of $\Gamma$ with coefficients in the left $\Z\Gamma$-module $L\Gamma$. Since the cohomology group inherits a right $L\Gamma$-module structure a dimension can be defined.
\begin{remark}
The usual definition of $\ell^2$-Betti numbers uses the group homology rather than the cohomology. Also, usually $\ell^2 G$ is used instead of $LG$. That the various definitions coincide was shown in \cite{pettho}.
\end{remark}

\subsection{L\"uck's Approximation Theorem}A striking result, due to L\"uck, states that for a finitely presented and residually finite group, the first $\ell^2$-Betti number is a normalized limit of ordinary Betti-numbers for a chain of subgroups of finite index, see \cite{lueckl2} for a proof. The result says more precisely:

\begin{theorem}[L\"uck] \label{lueck}
Let $G$ be a residually finite and finitely presented group. Let $\dots \subset H_{n+1} \subset H_n \subset \dots \subset G$ be a chain of finite index normal subgroups such that $\cap_{n=1}^{\infty} H_n =\{e\}$. Then,
$$\beta_1^{(2)}(G) = \lim_{n \to \infty} \frac{{\rm rk}((H_n)_{\rm ab})}{[G:H_n]} =\lim_{n \to \infty} \frac{\dim_\C \, H^1(G,\Z[G/H_n]) \otimes_{\Z} \C}{[G:H_n]}.$$
\end{theorem}

This result has numerous applications and extensions, we call it L\"uck's Approximation Theorem.

\subsection{Lower bounds on the first $\ell^2$-Betti number}

It is well known that the first $\ell^2$-Betti number of a finitely generated group $G$ is bounded from above by the number of generators of the group minus one. 
A more careful count reveals that a generator of order $n$ counts only $1 - \frac1n$.
Similarly, lower bounds can be found in terms of the order of the imposed relations in some presentation. More precesily, we find:
\begin{theorem} \label{posbetti} Let $G$ be an infinite 
countable discrete group. Assume that there exist subgroups $G_1,\dots,G_n$,
such that
$$G = \langle G_1, \dots , G_n \mid  r_1^{w_1},\dots,r_k^{w_k}, \dots \rangle,$$
for elements $r_1,\dots,r_k \in G_1 \ast \cdots \ast G_n$ and positive integers $w_1,\dots,w_k$.
We assume that the presentation is irredundant in the sense that $r_i^l \neq e \in G$, for $1<l < w_i$ and $1 \leq i < \infty$.
Then, the following inequality holds:
$$\beta_{1}^{(2)}(G) \geq n-1 + \sum_{i=1}^n \left( \beta_1^{(2)}(G_i) - \frac{1}{|G_i|}\right) - \sum_{j=1}^\infty \frac{1}{w_j}. $$
\end{theorem}

A proof of this result was given in \cite{pettho}. It can be used in many cases already if the groups $G_i$ are isomorphic to $\Z$ or $\Z/p\Z$, see for example \cite{thomosin}.

Another result says that the set of marked groups with first $\ell^2$-Betti number greater or equal so some constant is closed in Chabouty's space of marked groups, see \cite{pichot} for definitions and further references. More precisely, we have:

\begin{theorem}[Pichot, see \cite{pichot}]
Let $((G_n,S))_{n \in \N}$ be a convergent sequence of marked groups in Chabouty's space of marked groups. Then,
$$\beta_1^{(2)}(G) \geq \limsup_{n \to \infty} \beta_1^{(2)}(G_n).$$
\end{theorem}

This applies im particular to limits of free groups and shows that they all have a positive first $\ell^2$-Betti numbers.
In particular, there is an abundance of finitely presented groups with positive first $\ell^2$-Betti number. 

\section{Normal generation by torsion elements}
The first main result in this note extends the trivial upper bound from Equation \eqref{triv} (under some additional hypothesis) to the case where the group is \emph{normally} generated by a certain finite set of elements. The additional hypothesis is that the group $G$ be finitely presented and residually $p$-finite for some prime $p$. More precisely:

\begin{theorem} \label{main1}
Let $G$ be an infinite, finitely presented, residually $p$-finite group for some prime $p$. If $G$ is normally generated by a subgroup $\Lambda$, then
$$\beta_1^{(2)}(G) \leq \beta_1^{(2)}(\Lambda).$$
In particular, if $G$ is normally generated by elements $g_1,\dots,g_k \in G$ of order $n_1,\dots,n_k \in \{1,2,\dots \} \cup \{\infty\}$, 
then 
\begin{equation} \label{maineq}
\beta_{1}^{(2)}(G)  \leq k - 1 - \sum_{i=1}^k \frac1{n_i}.
\end{equation}
%Moreover, if we have equality, then the natural map $$\iota \colon \underbrace{\Z/n_1\Z \ast \cdots \ast \Z/n_k\Z}_{k} \to G$$ is injective. 
\end{theorem}

\begin{proof}
Let $H$ be a finite index normal subgroup of $G$, such that $G/H$ is of $p$-power order. We consider $Z^1(G,\Z[G/H])$, the abelian group of 1-cocycles of the group $G$ with values in the $G$-module $\Z[G/H]$. In a first step, we will show that the restriction map
$$\sigma \colon Z^1(G,\Z[G/H]) \to Z^1(\Lambda,\Z[G/H])$$
is injective.

Note that there is a natural injective evaluation map
$$\pi\colon Z^1(\Lambda,\Z[G/H]) \to \Z[G/H]^{\oplus k}$$
which sends a 1-cocycle $c$ to the values on the $g_i$, i.e. $c \mapsto (c(g_i))_{i=1}^k.$ 

We claim that $\pi \circ \sigma$ is injective. Indeed, assume that $c \in \ker(\pi \circ \sigma)$ and assume that $c(g) \in \omega^m$ for all $g \in G$, where $m$ is some integer greater than or equal zero. Since $g$ is in the normal closure of $g_1,\dots,g_n$, there exists some natural number $l \in \N$ and $h_1,\dots h_l \in G$, such that
$$g = \prod_{i=1}^l  h_i g^{\pm 1}_{q(i)} h_i^{-1},$$
for some function $q\colon \{1,\dots,l\} \to \{1,\dots,k\}$.
Computing $c(g)$ using the cocycle relation and $c(g^{\pm}_i)=0$, for $1 \leq i \leq k$, we get
$$c(g) = \sum_{i=1}^l \left(\prod_{j=1}^{i-1} h_j g^{\pm}_{q(j)} h_j^{-1} \right) (1-h_i g^{\pm}_{q(i)} h_i^{-1})c(h_i).$$
By hypothesis $c(h_i) \in \omega^m$ for $1 \leq i \leq l$ and we conclude that $c(g) \in \omega^{m+1}$. This argument applies to all $g \in G$.
Since the hypothesis $c(g) \in \omega^m$ is obviously satisfied for $m =0$, we finally get by induction that $c(g) \in \omega^m$ for all $ m\in \N$ and hence $$c(g) \in \bigcap_{m=1}^{\infty} \omega^m,\quad \forall g \in G.$$ By Theorem \ref{gr1} and since $G/H$ is of prime power order, we know that $\cap_{m=1}^{\infty} \omega^m = \{0\}$. Hence, $c(g)=0$ for all $g \in G$. This proves that the map $\pi\circ \sigma$ and hence $\sigma$ is injective.

By assumption, there exists a chain
$$\cdots \subset H_{n+1} \subset H_n \subset \cdots \subset G$$
of finite index subgroups with $p$-power index such that $$\bigcap_{n=1}^{\infty} H_n= \{e\}.$$

The claim is now implied by L\"uck's Approximation Theorem (see Theorem \ref{lueck}). Indeed, L\"uck's Approximation Theorem applied to the chain of finite index subgroups of $p$-power index gives:
\begin{equation} \label{eq2}
\beta_1^{(2)}(G) = \lim_{n \to \infty} \frac{\dim_\C \, H^1(G,\C[G/H_n])}{[G : H_n]} \leq \lim_{n \to \infty} \frac{\dim_\C \, H^1(\Lambda,\C[G/H_n])}{[G : H_n]} \leq \beta_1^{(2)}(\Lambda).
\end{equation}
Here, we used Kazhdan's inequality in the last step, see \cite{osl} for a proof.
This finishes the proof of the first inequality. The second claim follows from a simple and well-known estimate for the first $\ell^2$-Betti number. Indeed, for each $H$, we can estimate the dimension of the image of $\pi \otimes_{\Z} \C$.
Since $g_i$ has order $n_i$, we compute $$0=c(g_i^{n_i}) = \left(\sum_{j=0}^{n_i-1} \bar g_i^j\right) c(g_i),$$
where we denote by $\bar g_i$ the image of $g_i$ in $G/H$.
If the order of the image of $\bar g_i$ is $m_i$, then $n_i^{-1}\sum_{j=0}^{n_i-1} \bar g_i^j$ is a projection of normalized trace $m_i^{-1}$ such that $$\dim_{\C} \, \left(\im(\pi)\otimes_{\Z} \C\right) \leq [G:H] \cdot \sum_{i=1}^k \left(1 - \frac {1}{m_i} \right) \leq [G:H] \cdot \sum_{i=1}^k \left(1 - \frac {1}{n_i} \right).$$

This implies that
$$\dim_\C \, H^1(G,\Z[G/H]) \otimes_{\Z} \C \leq [G:H] \cdot \sum_{i=1}^k \left(1 - \frac {1}{n_i} \right) -[G :H] +1 .$$ This finishes the proof, again using L\"uck's Approximation Theorem.
\end{proof}

\begin{remark} \label{lack}
Let $G$ be an infinite, residually $p$-finite group.
It follows from Proposition 3.7 in \cite{MR2271299} in combination with L\"uck's Approximation Theorem that $$\beta_1^{(2)}(G) \leq \dim_{\Z/p\Z} H^1(G,\Z/p\Z) -1.$$
This implies that $\beta_1^{(2)}(G) \leq k-1$ in the situation that $G$ is normally generated by $g_1,\dots,g_k$. Our result improves this estimate in the case when some of the elements $g_1,\dots,g_k$ have finite order.
\end{remark}

\begin{remark} \label{rem1}
Consider $G=PSL(2,\Z) = \langle a,b \mid a^2=b^3=e\rangle$. Then, $\beta_1^{(2)}(G)= \frac1{6} \neq 0$ and $G$ is normally generated by the element $ab \in G$. Hence, the assumption that $G$ be residually a $p$-group cannot be omitted in Theorem \ref{main1}.
\end{remark}

\section{An uncertainty principle and applications}
\label{uncertsec}
In this section we want to prove a quantitative estimate on the girth of a marked group in terms of its first $\ell^2$-Betti number. In \cite{thomosin}, we constructed for given $\varepsilon>0$ a $k$-generated simple groups with first $\ell^2$-Betti number greater than $k-1-\varepsilon$. The construction involved methods from small cancellation theory and in particular, those groups did not admit any short relations in terms of the natural generating set. This in fact follows already from the main result in \cite{pichot}. If $(G_i,S_i)_{i \in \N}$ is a sequence of marked groups with $|S_i|= k$ and $\lim_{i \to \infty} \beta_1^{(2)}(G_i) = k-1$, then necessarily $$\lim_{i \to \infty}{\rm girth}(G_i,S_i) = \infty,$$
where ${\rm girth}(G,S)$ denotes the length of the shortest cycle in the Cayley graph of $G$ with respect to the generating set $S$.
 Indeed, by \cite{pichot}, any limit point $(G,S)$ of the sequence $(G_i,S_i)_{i \in \N}$ satisfies $\beta_1^{(2)}(G)=k-1$ and hence is a free group on the basis $S$. (This last fact is well-known and is also a consequence of our next theorem.)
In this section, we want to prove a quantitative version of this result.

\begin{theorem} \label{ineq}
Let $G$ be a finitely generated group with generating set $S=\{g_1,\dots,g_k\}$. Then,
$${\rm girth}(G,S) \geq \frac1{k-1-\beta_1^{(2)}(G)}.$$
\end{theorem}

In order to prove this theorem, we need some variant of the so-called uncertainty principle. We denote by $\|.\|$ the usual operator norm on $B(\ell^2G)$ and use the same symbol to denote the induced norm on $\C G$, i.e., $\|f\| = \|\lambda(f)\|$ for all $f \in \C G$. The 1-norm is denoted by $\|\sum_g a_g g\|_1 = \sum_g |a_g|$. For $f = \sum_g a_g g$ we define its support as ${\rm supp} := \{g \in G \mid a_g \neq 0\}$.

\begin{theorem} \label{uncert}
Let $G$ be a group and $f \in \mathbb \C G$ be a non-zero element of the complex group ring. Then,
$$\dim_{LG} (f \cdot LG) \cdot |{\rm supp}(f)| \geq 1.$$
\end{theorem}
\begin{proof}
First of all we have $\dim_{LG}(f \cdot LG) = \tau(p_K)$, where $K$ is the closure of the image of $\lambda(f) \colon \ell^2 G \to \ell^2 G$.
\begin{equation} \label{eq1}
\tau(f^*f) \leq\dim_{LG}(f \cdot LG) \cdot \|f\|^2 
\end{equation}
since
$$\tau(f^*f) =\tau(ff^*) = \tau(p_Kff^*) \leq \tau(p_K) \cdot \|ff^*\| = \dim_{LG} (f \cdot LG) \cdot \|f\|^2.$$

Secondly, we see that
\begin{equation} \label{eq2}
\|f\|^2_1 \leq |{\rm supp} (f) | \cdot \tau(f^*f)
\end{equation}
by the Cauchy-Schwarz inequality applied to $f\cdot \chi_{{\rm supp} (f)}$, where the product is here the pointwise product of coefficients and $\|f\|_1$ denotes the usual 1-norm on $\mathbb C[G]$. Combining Equations \eqref{eq1} and \eqref{eq2} we conclude
\begin{equation}
\dim_{LG} (f \cdot LG) \cdot |{\rm supp}(f)| \geq \left(\frac{\|f\|_1}{\|f\|} \right)^2.
\end{equation}

Now, since each group element acts as a unitary, and hence with operator norm $1$ on $\ell^2G$, we get $\|f\|_1 \geq \|f\|$. This proves the claim and finishes the proof.
\end{proof}

The preceding result and the following corollary were proved as result of a question by Efremenko on MathOverflow.

\begin{corollary}
Let $G$ be a finite group and $f \in \C[G]$ be an arbitrary non-zero element. Then,
$$\dim_{\C} (f \cdot \C[G]) \cdot |{\rm supp}(f)| \geq |G|.$$
\end{corollary}

We are now ready to prove Theorem \ref{ineq}.
\begin{proof}[Proof of Theorem \ref{ineq}:]
Again, we study the map $\pi \colon Z^1(G,LG) \to LG^{\oplus k}$, which is given by $c \mapsto (c(g_i))_{i=1}^k$. If $w \in {\mathbb F}_k$ is some word such that $w(g_1,\dots,g_k)=e$ in $G$, then
$$0 = c(w(g_1,\dots,g_k)) = \sum_{i=1}^k \partial_i(w)(g_1,\dots,g_k) \cdot c(g_i),$$
$\partial_i \colon \Z[\mathbb F_k] \to \Z[\mathbb F_k]$ denotes the $i$-th Fox derivative for $1 \leq i \leq k$. Thus, the image of $\pi$ lies is annihilated by the $LG$-linear map $(\xi_1,\dots,\xi_k) \mapsto \sum_{i=1}^k \partial_i(w)(g_1,\dots,g_k) \xi_i$. 
In particular, the image of $\pi$ does not intersect with the kernel of $\partial_i(w)$ in the $i$-th coordinate. The number of summands in $\partial_i(w)$ is equal to the number of occurances of the letters $g_i^{\pm}$ in $w$. Thus, we have
$$\sum_{i=1}^k |{\rm supp}(\delta_i(w))| = \ell(w)$$
and $$\im(\pi) \subseteq \bigoplus_{i=1}^k  (LG \ominus \im(\partial_i(w))).$$
Thus $$\dim_{LG} \im(\pi) \leq k - \sum_{i=1}^k |{\rm supp}(\delta_i(w))|^{-1} \leq k - \frac1{\ell(w)}.$$
This implies that
$\beta^{(2)}_1(G) \leq k-1 - 1/\ell(w)$ and finishes the proof. \end{proof}

\section*{Acknowledgments}
I want to thank Klim Efremenko, Mikhael Ershov, Ana Khukhro, and Denis Osin for interesting comments. This research was supported by ERC-Starting Grant No.\!\! 277728 {\it Geometry and Analysis of Group Rings}.

\end{document}